\documentclass{amsart}
\usepackage{epsf,graphicx,amsmath,verbatim,epsfig}

\newcommand{\Ind}{\mathrm{Ind}}
\newcommand{\Inf}{\mathrm{Inf}}
\newcommand{\GST}{\mathrm{GST}}
\newcommand{\bx}{\mathbf{x}}
\newcommand{\by}{\mathbf{y}}

\newtheorem{lemma}{\sc Lemma}[section]
\newcommand{\RR}{{\rm I\kern -1.6pt{\rm R}}}

\title[The Generalized Symmetric Tequila Problem]{The Generalized Symmetric Tequila Problem: Influence and Independence in N-Player Games}
\author[]{}

\email{}
 
\begin{document}
\maketitle

\centerline{\scshape   Denali Molitor}
\begin{abstract}
This paper extends results from Mike Steel and Amelia Taylor's paper \emph{The Structure of Symmetric N-Player Games when Influence and Independence Collide}. 
These games include $n$ causes, which are dichotomous random variables whose values determine the probabilities of the values of $n$ dichotomous effects. 
We denote the probability spaces that exhibit independence and influence among $n$ players as $\Ind_n$ and $\Inf_n$ respectively. 
We define the solution space of the `generalized symmetric tequila problem,' $\GST_n$, as the set of probabilities for a set of given effects such that the causes and effects are independent and each cause influences the effects, i.e.
\begin{equation}
\GST_n=\Ind_n\cap \Inf_n.
\end{equation}
Steel and Taylor showed that $\GST_n$ is connected for $n\geq 8$ and disconnected for $n=3,4$. We prove that for $n=5,6,7$, $\GST_n$ is connected and determine the number of connected components of $\GST_4$.

\end{abstract}
 
\section{Introduction}
In this paper we examine the intersection of independence and influence in $n$-player games, specifically addressing open questions from \cite{taylor} and providing additional examples of scenarios for which this model may be useful. In these games, we consider $n$ causes, dichotomous random variables, whose values determine the probabilities of the values of $n$ dichotomous effects. Independence is characterized by the requirement that $\mathbb{P}(E_i\wedge E_j|C_k)=\mathbb{P}(E_i|C_k)\mathbb{P}(E_j|C_k)$. 
We say that a set of $k$ causes influences a set of $m$ effects if for each pair of a cause $C_i$ and an effect $E_j$, there
exists at least one assignment of states for the remaining $k-1$ causes, such that some change
in the state of $C_i$, while holding the values of the remaining $k-1$ causes fixed, changes the
probability of the effect $E_j$ \cite{taylor} \cite{sober}. 
As in \cite{taylor}, we consider the case in which $\mathbb{P}(C_i)= \frac{1}{2}$ for all $i$.
We denote the spaces that exhibit independence and influence among $n$ players as $\Ind_n$ and $\Inf_n$ respectively. We define the solution space of the `generalized symmetric tequila problem,' $\GST_n$, as 
\begin{equation}
\GST_n=\Ind_n\cap \Inf_n.
\end{equation}
In \cite{taylor}, Steel and Taylor show that for $n\geq 8$, $\GST_n$ is connected, while $\GST_3$ consists of two connected components and $\GST_4$ is disconnected. In this paper, we argue that $\GST_n$ is connected for $n=5,6,7$ and show that $\GST_4$ consists of two connected components. Due to the relationship with \cite{taylor}, we rely heavily on definitions and theorems proved by Steel and Taylor.

\section{Background Information}
In order to study the connectedness of $\GST_n$ we first characterize the sets $\Ind_n$ and $\Inf_n$. Proposition 3.2 of \cite{taylor} states that a point $\bx\in\Ind_n$ if and only if
\begin{equation}
\Psi(\bx)=\left(\frac{1}{2^{n-1}}\sum_{k=0}^{n-1}{n-1 \choose k}x_{k+1}\right)^2-\frac{1}{2^{n-1}}\sum_{k=0}^{n-2}{n-2 \choose k}(x_{k+2}^2+x_{k+1}x_{n-(k+1)})=0,
\end{equation}
which expresses the property that
\begin{equation}
P(E_i\wedge E_j | C_k)-P(E_i|C_k)P(E_j|C_k)=0.
\end{equation}
By Proposition 3.1 of \cite{taylor}, $\bx\in \Inf_n$ if and only if there exists $s\in [n]$ such that $x_s\neq x_{n-s+1}$. Since $\GST_n$ can be equivalently defined as $\Ind_n- \Inf_n^c$, it is useful to consider the complement of $\Inf_n$, 
\begin{equation}
\Inf_n^c=\{\bx\in\mathbb{R}|x_i=x_{n-i+1}, \, \forall i\in [n]\},
\end{equation}
in determining the connectedness of $\GST_n$.

Steel and Taylor \cite{taylor} also showed that if $Q_n$ is the matrix corresponding to the quadratic form $\Psi$, then there exists a diagonal matrix $D$ of real eigenvalues of $Q$, and a real orthogonal matrix $P$ such that $P^TQ_nP=D$. If $\by=P^T\bx$, then 
 solutions to $\by^TD\by=0
$ are of the form
\begin{equation}
S_{s,t}=\left\{\by\in\mathbb{R}^n\bigg|-s<y_1<s,\sum_{i=2}^k\lambda_i y_i^2=t \text{ and } \sum_{i=n-k+1}^n (-\lambda_i) y_i^2=t \right\},
\end{equation}
for some $s$ and $t$, where $k-1$ is the number of positive eigenvalues and $n-k-1$ is the number of negative eigenvalues. Note that $t\geq 0$ and $s > 0$, for if $s\leq 0$, then $S_{s,t}$ is empty. We therefore assume $t\geq 0$ and $s>0$ throughout the remainder of the paper.
Also, as in \cite{taylor}, we define $T$ to be the set of all $\by$ such that the corresponding $\bx\in\Inf_n^c$, that is
\begin{equation}
T=\left\{\by=P^T\bx|\bx\in\Inf_n^c\right\}.
\end{equation}
Applying $P^T$ to $\Ind_n$ and $\Inf_n^c$ gives the spaces $S_{s,t}$ and $T$ respectively, within which studying connectedness is easier. The fact that $P^T$ is a homeomorphism follows directly from the fact that $P$ is orthogonal. We use this fact repeatedly and so state it as a Lemma. 

\begin{lemma} The map $P^T$ is a homeomorphism.
\end{lemma}

Since $P^T$ is a homeomorphism, from $\Ind_n\cap\Inf_n$ to  $(\cup S_{s,t\geq0})-T$, these two spaces must have the same number of connected components. The structures of $S_{s,t\geq0},$ and $T\cap  S_{s,t\geq0}$ are simpler to study and provide clues as to the structure of $(\cup S_{s,t\geq0})-T$. We examine the connectedness of $\GST_n$ by considering the connectedness of the $S_{s,t\geq0},$ $T\cap S_{s,t\geq0}$ and finally $(\cup S_{s,t\geq0})-T$.\\

In the following sections, we give a case-by-case study of the connectedness of $\GST_n$ for $n=4,5,6,7$. Our first step in each case is to compute the dimension of $\Ind_n\cap \Inf_n^c$ and apply $P^T$ to this resulting space. We then examine $S_{s,t\geq0}\cap T$ in order to determine whether $S_{s,t\geq0}-T$ is connected. 
Showing that each of the $S_{s,t\geq0}-T$ are connected for $t>0$, is in fact, enough to show that $(\cup S_{s,t\geq0})-T$ is connected.

\begin{lemma}If $S_{s,t>0}-T$ is connected for all $s > 0$ and $t>0$, then for $n\geq 3$, $\GST_n$ is connected also.
\end{lemma}
\begin{proof} 
Let $\mathbf{m}=\left(\frac{1}{2},\frac{1}{2},\hdots,\frac{1}{2}\right)$. Let $s,t'>0$ be sufficiently small so that $\mathbf{m}+P\bx\in [0,1]^n$ for all $\bx\in S_{s,t'}$. 
Consider any points $\mathbf{p}$ and $\mathbf{q}\in\GST_n$.
Recall that for any point $\bx$ in $ \GST_n$, $\by=P^T \bx$ satisfies
\[\sum_{i=2}^k\lambda_i y_i^2=M \text{ and } \sum_{i=n-k+1}^n (-\lambda_i) y_i^2=M,\text{ for some }M. \]
 Now, suppose that for some $\by$, $M = 0$, then $\by$ must be of the form $\by=(y_1,0,0,\hdots,0)$ with $P\by = (1,1,1,\hdots,1)$. Such $\by$, however, are in $T$ and have $P\by \in \Inf_n^c$, thus failing the influence condition. Therefore, $\by\in(\cup S_{s,t}) - T$ requires $M>0$.
Denote the $M$ corresponding to $\mathbf{p}$ and $\mathbf{q}$ as $M_p$ and $M_q$ respectively.
Since $M_p>0$ and $M_q>0$, we can choose $c_1=\frac{t}{M_p}$ and $c_2=\frac{t}{M_q}$ for some $t\in(0,t']$.
Then for $c_1\by_p=c_1P^T\mathbf{p}$ and $c_2\by_q=c_2P^T\mathbf{q}$, we have $ c_1\by_p, c_2\by_q\in S_{s,t}$. 
 Since $S_{s,t}-T$ is connected for all $t>0$, there exists a path from $\by_p$ to $\by_q$ in $S_{s,t}-T$. By the fact that $P^T$ is a homeomorphism, there also exists a continuous path from $P\by_p$ to $P\by_q$ satisfying $\Ind_n$ and $\Inf_n$, but not necessarily within probability space, $[0,1]^n$. In order to ensure that there is a path in probability space as well, we scale the path from $P\by_p$ to $P\by_q$ by adding $\mathbf{m}=\left(\frac{1}{2},\frac{1}{2},\hdots, \frac{1}{2}\right)$ to the entire path. 
For small enough $t$ and $t'$, this path from $\mathbf{m}+P\by_p$ to $\mathbf{m}+P\by_q$ remains in $[0,1]^n$ and hence is in $\GST_n$.
 Note that $P\by_p=P(c_1P^T\mathbf{p})$. 
Also, using Theorem 6.2 and Remark 6.3 from \cite{taylor}, we know that the straight-line paths from
$P\by_p$ to $\mathbf{m}+P\by_p$ and $P\by_q$ to $\mathbf{m}+P\by_q$,
remain in $[0,1]^n$ and are in $\GST_n$ as well.
Therefore, if $S_{s,t}-T$ is connected for all $t>0$, then $\GST_n$ is connected also.
\end{proof}

In the previous Lemma, we were only concerned with $S_{s,t}$ for $t>0$. We need not consider $t =0$, since for all $s$, $S_{s,0}\subset T$, meaning none of the corresponding $\bx$ satisfy influence and are therefore not in $\GST_n$. 
In the following sections, we show that for $n>4$, $S_{s,0}\subset T$ does not disconnect $\underset{s,t\geq0}\cup S_{s,t}$ when removed.

\section{$\GST_4$}
We first examine the space $\Ind_4\cap \Inf_4^c$ by setting $\Psi(\bx)=0,\, x_1=x_4$ and $x_2=x_3.$ Beginning with $\Psi$ for $n=4$,

\begin{align*}
&\Psi(\bx)
=\left(\frac{1}{2^{3}}\sum_{k=0}^{3}{3 \choose k}x_{k+1}\right)^2-\frac{1}{2^{3}}\sum_{k=0}^{2}{2 \choose k}(x_{k+2}^2+x_{k+1}x_{4-(k+1)}).\\
\end{align*}
\begin{align*}
&\text{ Applying } x_1=x_4\text{ and }x_2=x_3,\\
&\Psi(\bx)=\frac{1}{64}\left(2x_{1}+6x_{2}\right)^2-\frac{1}{8}\left(x_{2}^2+2x_{1}x_{3}+2x_{3}^2+2x_{2}^2+x_{4}^2\right).\\
&\text{Since $\Psi(x)=0$, we multiply through by 8 and find},\\
&0=\frac{1}{2}x_{1}^2+3x_1x_2+\frac{9}{2}x_{2}^2-\left(5x_{2}^2+2x_{1}x_{2}+x_{1}^2\right)\\
&\quad=-\frac{1}{2}x_{1}^2+x_1x_2-\frac{1}{2}x_{2}^2\\
&\quad=-\frac{1}{2}(x_{1}-x_{2})^2.\\
\end{align*}
Since $-\frac{1}{2}(x_{1}^2-x_{2}^2)^2\leq 0$ for all real $x_1,x_2$, the only solution to this system requires that $x_1=x_2$. Therefore, if $\bx\in \Ind_4\cap \Inf_4^c$, then 
 $x_1=x_2=x_3=x_4$. 
Hence $\Ind_4\cap \Inf_4^c$ forms a one-dimensional linear space with basis vector $(1,1,1,1)$. The result of applying $P^T$ to this space gives $S_{s,t}\cap T$, a one-dimensional space with basis $(1,0,0,0)$. 
We now explore the connectedness of 
$S_{s,t}-T$ and consequently the connectedness of $\GST_4$.

Recall the definition of $S_{s,t}$ given in Equation 4.
For $n=4,$
\[S_{s,t}=\{\by\in\mathbb{R}^n|-s<y_1<s,\lambda_2 y_2^2=t \text{ and } (-\lambda_3) y_3^2+(-\lambda_4)y_4^2=t \}.\]
Since any $\by\in T\cap \left( \cup S_{s,t}\right)$ is of the form $(y_1,0,0,0)$, then $S_{s,0} = T\cap \left( \cup S_{s,t}\right)$ and $T\cap S_{s,t>0}=\emptyset$. Therefore, when examining connectedness of the transformed $\GST_4$ space, $ S_{s,t>0}-T=S_{s,t>0}$ and we need only consider the connectedness of the $S_{s,t>0}$. For any $\by\in S_{s,t>0}$, $y_2=\pm\sqrt{\frac{t}{\lambda_2}}$, and $S_{s,t>0}=I_s\times S^0\times S^1$ is a two-dimensional space.
If we consider any $\mathbf{p}\in \cup S_{s,t>0}$, where $p_2<0$ and any $\mathbf{q}\in \cup S_{s,t>0}$, where $q_2>0$, then there does not exist a continuous path from $\mathbf{p}$ to $\mathbf{q}$ in $S_{s,t>0}$ by the Intermediate Value Theorem, which implies that such a path must contain a point $\by$ with  $y_2=0$.
However, $y_2 = 0$ implies $\by \in S_{s,t = 0}=T$ and $\by \not\in S_{s,t>0}$. Thus $\GST_4$ consists of at least two components. Observe that each $S_{s,t>0}$ forms two disjoint cylinders of the form $I_s\times\sqrt{\frac{t}{\lambda_2}}\times S^1$ and $I_s\times-\sqrt{\frac{t}{\lambda_2}}\times S^1$ \cite{taylor}. Each cylinder forms a connected space. The two cylinders are disconnected by $(y_1,0,0,0)$.
Let $A_1$ denote $I_s\times \sqrt{\frac{t}{\lambda_2}}\times S^1$ and $A_2$ denote $I_s\times -\sqrt{\frac{t}{\lambda_2}}\times S^1$. 
Then each $S_{s,t}\cap A_1$ and $S_{s,t}\cap A_2$ is connected. The argument used in Lemma 2.2 now implies that $\underset{s,t}\cup A_1$ and $\underset{s,t}\cup A_2$ are each connected and therefore $\GST_4$ consists of two connected components.

\section{The connectedness of $\GST_5$}
If we set $\Psi(\bx)=0$, include the conditions that $x_1=x_5, x_2=x_4$ and simplify, we see \begin{align*}
\Psi(\mathbf{x})&=\Bigg(\frac{1}{16}\sum_{k=0}^4 {4\choose k}x_{k+1}\Bigg)^2-\frac{1}{16}\sum_{k=0}^3{3\choose k}(x_{k+2}^2+x_{k+1}x_{n-(k+1)})\\
&=\frac{1}{2^8}[2x_{1}+8x_{2}+6x_{3}]^2-\frac{1}{16}\Big[(x_{2}^2+x_{1}x_{2})+3(x_{3}^2+x_{2}x_{3})+3(x_{2}^2+x_{3}x_{2})+(x_{1}^2+x_{2}x_{1})\Big]\\
&=0.
\end{align*}
We mulitiply through by 64 and find,
\begin{align*}
0&=[x_{1}^2+16x_{2}^2+9x_{3}^2+8x_1x_2+6x_1x_3+24x_2x_3]-
4\Big[4x_{2}^2+2x_{1}x_{2}+3x_{3}^2+6x_{2}x_{3}+x_{1}^2\Big]\\
&=-3x_{1}^2-3x_{3}^2+6x_1x_3\\
&=-3(x_{1}-x_{3})^2.\\
\end{align*}
Again, the only solution to this equation occurs when $x_1=x_3$.
Thus, if $\bx \in\Ind_5\cap \Inf_5^c$ then $x_1=x_3=x_5$ and $x_2=x_4$. This intersection forms a two-dimensional linear space that can be written in terms of $x_1$ and $x_2$. Applying the matrix $P^T$ to this space also results in a two-dimensional space with basis vectors $\mathbf{b_1}=(1,0,0,0,0)$ and $\mathbf{b_2}=(0,0.686556,-0.606862,0.519196,-0.301323)$. 
Therefore any $\by\in\left(\cup S_{s,t}\right)\cap T$ (defined in Equations 5 and 6) is a linear combination of these two basis vectors and we write
\[\left(\cup S_{s,t}\right)\cap T=\{\by\in\mathbb{R}^5|\by=\mathbf{b_1}a+\mathbf{b_2}c \text{ for }a,c\in \mathbb{R}\}.\]
Since $\mathbf{b_1}=(1,0,0,0,0)$, we can equivalently write
\[\left(\cup S_{s,t}\right)\cap T=\{\by\in\mathbb{R}^5|\by=(y_1,b_{22}\cdot c,b_{23}\cdot c,b_{24}\cdot c, b_{25}\cdot c),c\in \mathbb{R}\}.\]

Recall that by Lemma 2.2, if $S_{s,t>0}-T$ is connected for each $s, t>0$, then $\cup S_{s,t>0}-T$ is connected also. Fix an $s$ and $t > 0$. Then for $\by\in S_{s,t}$, 
\[\lambda_2 y_2^2+\lambda_3 y_3^2=t = \lambda_4 y_4^2+\lambda_5 y_5^2.\]
Choosing the first of the two equivalent equations for $t$, we substitute $cb_{22}$ in for $y_2$ and $cb_{23}$ in for $y_3$ and get
\begin{equation}
t=\lambda_2(b_{22}\cdot c)^2+\lambda_3(b_{23}\cdot c)^2.
\end{equation}
Solving for $c$,
\begin{equation}
c=\pm \sqrt{\frac{t}{\lambda_2b_{22}^2+\lambda_3 b_{23}^2}}.
\end{equation}
We can now characterize $T\cap S_{s,t}$ as line segments of the form, 
\[T\cap S_{s,t>0}=\{\by\in\mathbb{R}^5|\by=(y_1,\pm c\cdot b_{22},\pm c\cdot b_{23},\pm c\cdot b_{24},\pm c\cdot b_{25}), -s<y_1<s\}.\]

Each $S_{s,t>0}$ is homeomorphic to $I_s\times S^1\times S^1$ \cite{taylor}, a three-dimensional path-connected space. We note that the two $S^1$ are circles based on the restrictions that $\lambda_2 y_2^2+\lambda_3 y_3^2=t \text{ and } \lambda_4 y_4^2+\lambda_5 y_5^2=t.$ We now show that $S_{s,t>0}-T$ is a path-connected space as well.

We observe that if $\mathbf{p}$ is any point in $S_{s,t>0}-T$
then for some $p_i$ with $2\leq i\leq 5$, $p_i \neq \pm cb_{2i}$ (and
hence $p_j\neq \pm cb_{2j}$), where $i\neq j$ and $i,j\in\{2,3\}$ or $i,j\in\{4,5\}$. Let
$\mathbf{q}$ be any point in $S_{s,t>0}$ such that $q_i = p_i$ (and
hence $q_j = p_j$).  Then
$\mathbf{q}\in S_{s,t>0}-T$ and, since $S^1$ is path-connected, there is a path in $S_{s,t>0}$ from
$\mathbf{p}$ to $\mathbf{q}$ where every point on the path has $p_i$
as the $i^{th}$ coordinate. Hence the entire path is actually in
$S_{s,t>0}-T$.   

Now let $\mathbf{p}$ and $\mathbf{q}$ be any two points in
$S_{s,t>0}-T$.  As before, for some $2\leq i\leq 5$, $p_i \neq \pm
cb_{2i}$. Without loss of generallity suppose 
$p_2\neq \pm cb_{22}$.  We split the argument into two cases, one in
which $q_4 = \pm cb_{24}$ and $q_5=\pm cb_{25}$ and the second with
$q_4 \neq \pm cb_{24}$ and $q_5\neq \pm cb_{25}$.  

If $q_4 = \pm cb_{24}$ and $q_5=\pm cb_{25}$, then $q_2\neq \pm cb_{22}$ and $q_3\neq \pm cb_{23}$.  
Since $S^1$ is continuous and $S_{s,t>0}\cap T$ is discrete, there
exist $s_4, s_5$ such that $\lambda_4s_4^2+ \lambda_5s_5^2 = t$ and  $s_4
\neq \pm cb_{24}$ and $s_5\neq \pm cb_{25}$.  Then, by the argument
above, there is a path in $S_{s,t>0}-T$ from 
$\mathbf{p}$ to $(p_1, p_2, p_3, s_4, s_5)$ and  from 
$(p_1, p_2, p_3, s_4, s_5)$ to $(q_1, q_2, q_3, s_4, s_5)$.  
Since $q_2\neq \pm cb_{22}$, there exists a path
from $(q_1, q_2, q_3, s_4, s_5)$ to $\mathbf{q}$ in $S_{s,t>0}-T$.  These paths combine
to give a path from $\mathbf{p}$ to $\mathbf{q}$ in $S_{s,t>0}-T$.

In the second case, $q_4 \neq \pm cb_{24}$ and 
$q_5\neq \pm cb_{25}$. Using the argument above, there is a path from 
$\mathbf{p}$ to $(p_1, p_2, p_3, q_4, q_5)$ and a path from 
$(p_1, p_2, p_3, q_4, q_5)$ to $\mathbf{q}$ both of which are in
$S_{s,t>0}-T$. 

We have now shown that $S_{s,t}-T$ is path connected since there exists a path between any $\mathbf{p}$ and $\mathbf{q}$ in $S_{s,t}-T$.
 Then $\GST_5$ is connected by Lemma 2.2.

\section{The connectedness of $\GST_6$}
If we set $\Psi(\bx)=0$, include the conditions that $x_1=x_6, x_2=x_5,x_3=x_4$ and simplify we find
\begin{align*}
&\Psi(\mathbf{x})=\Bigg(\frac{1}{2^5}\sum_{k=0}^5 {5\choose k}x_{k+1}\Bigg)^2-\frac{1}{2^5}\sum_{k=0}^4{4\choose k}(x_{k+2}^2+x_{k+1}x_{n-(k+1)})\\
&=\frac{1}{2^{10}}\Big[x_{1}+5x_{2}+10x_{3}+10x_{3}+5x_{2}+x_{1}\Big]^2-\\
&\quad-\frac{1}{2^5}\Big[(x_{2}^2+x_{1}x_{2})+4(x_{3}^2+x_{2}x_{3})+6(x_{3}^2+x_{3}x_{3})+4(x_{2}^2+x_{3}x_{2})+(x_{1}^2+x_{2}x_{1})\Big].\\
\end{align*}
We mulitiply through by $2^8$,
\begin{align*}
0&=x_{1}^2+25x_{2}^2+100x_{3}^2+10x_1x_2+20x_1x_3+100x_2x_3-
8\Big[x_1^2+ 5x_{2}^2+2x_{1}x_{2}+16x_{3}^2+8x_{2}x_{3}\Big]\\
&=-7x_{1}^2-15x_{2}^2-28x_{3}^2-6x_1x_2+20x_1x_3+36x_2x_3.\\
\end{align*}
Solving for $x_1$, we find
\begin{align*}
x_1&=\frac{1}{7}\bigg(-3x_2+10x_3\pm\sqrt{-(x_2-x_3)^2}\bigg).
\end{align*}
The only real solution to this equation occurs when $x_2=x_3$. Making this substitution,
\begin{align*}
x_1&=\frac{1}{7}\big(-3x_2+10x_2\big)=x_2.
\end{align*}
Now, all real solutions must have $x_2=x_3$, $x_1=x_2$ and therefore $x_1=x_2=x_3=x_4=x_5=x_6$. 
Thus, $\Ind_n\cap \Inf_6^c$ forms a one-dimensional linear space with basis vector $(1,1,1,1,1,1)$. The result of applying $P^T$ to this space gives the one-dimensional space with basis $(1,0,0,0,0,0)$. Then for $n=6$, \[T=\{\by|\by=(y_1,0,0,0,0,0)\}.\] 
Now, $T\subseteq S_{s,0}$ since $\lambda_2 y_2^2+\lambda_3 y_3^2=0 \text{ and } \lambda_4 y_4^2+\lambda_5 y_5^2+\lambda_6 y_6^2=0$ for all $\by\in T$. Thus, for all $s,t>0$, $ S_{s,t}-T =  S_{s,t}$ when $n=6$. 
Now $S_{s,t}$ is homeomorphic to $I_s\times S^1\times S^2$, where $I_s$ is the interval corresponding to $s$, $S_1$ is a circle and $S^2$ is a sphere. Each of these spaces are connected, and the product of these spaces, $S_{s,t}$ is connected as well. 
Then by Lemma 2.2, since each of the $S_{s,t}$ are connected for $t>0$, $\GST_6$ is connected as well.

 \section{The connectedness of $\GST_7$}
If we set $\Psi(\bx)=0$, include the condition that $x_1=x_7, x_2=x_6,x_3=x_5$ and simplify we find
\begin{align*}
&\Psi(\mathbf{x})=\Bigg(\frac{1}{2^6}\sum_{k=0}^6 {6\choose k}x_{k+1}\Bigg)^2-\frac{1}{2^6}\sum_{k=0}^5{5\choose k}(x_{k+2}^2+x_{k+1}x_{n-(k+1)})\\
&=\frac{1}{2^{12}}\Bigg[2x_{1}+12x_{2}+30x_{3}+20x_{4}\Bigg]^2-\frac{1}{2^6}\Bigg[x_1^2+6x_{2}^2+15x_{3}^2+10x_{4}^2+2x_{1}x_{6}+10x_{2}x_{5}+20x_{4}x_{3}\Bigg]\\&=0.
\end{align*}
Mulitiplying through by $2^{10}$ and simplifying,
\begin{align*}
&-15x_{1}^2-60x_{2}^2-15x_{3}^2-60x_{4}^2-20x_1x_2+30x_1x_3+20x_1x_4+20x_2x_3+120x_2x_4-20x_3x_4=0\\
\end{align*}
Solving for $x_4$, we find,
\[x_4=\frac{1}{6}\bigg(x_1+6x_2-x_3-2\sqrt{2}\sqrt{-(x_1-x_3)^2}\bigg).\]
Thus, all real solutions must have $x_1=x_3$.
Hence $\bx\in \Ind_7\cap \Inf_7^c$, implies 
$x_1=x_3=x_5=x_7$ and $x_2=x_4=x_6$, which forms a two-dimensional linear space. Applying $P^T$ transform, we find that $T$ is a two-dimensional linear space with basis vectors $\mathbf{b_1}=(1,0,0,0,0,0)$ and $\mathbf{b_2}=(0,-0.6902,0.6635,-0.5705,0.5168,-0.3974,0.2172)$.
Fix $s,t>0$. As in the $n=5$ case, we take one of the two equivalient expressions for $t$ and  substitute $\mathbf{b_2}$ in for $\by$ to get, 
\begin{equation}
t=\lambda_2(b_{22}\cdot c)^2+\lambda_3(b_{23}\cdot c)^2+\lambda_4(b_{24}\cdot c)^2.
\end{equation}
We now characterize $c$ in terms of $t$,
\begin{equation}
c=\pm \sqrt{\frac{t}{\lambda_2b_{22}^2+\lambda_3 b_{23}^2+\lambda_4 b_{24}^2}}.
\end{equation}
As in the $n = 5$ case, we can represent $T\cap S_{s,t}$ as the union of lines of the form, 
\[T\cap S_{s,t>0}=\{\by\in\mathbb{R}^5|\by=(y_1,\pm c\cdot b_{22},\pm c\cdot b_{23},\pm c\cdot b_{24},\pm c\cdot b_{25},\pm c\cdot b_{26},\pm c\cdot b_{27}), -s<y_1<s\}.\]
When $n=7$, $S_{s,t}$ forms a five dimensional space homeomorphic to $I_s\times S^2\times S^2$. We now show that the removal of $T$ cannot disconnect these space using a proof similar to that used for $\GST_5$.

We observe that if $\mathbf{p}$ is any point in $S_{s,t>0}-T$
then for some $p_i$ with $2\leq i\leq 7$, $p_i \neq \pm cb_{2i}$, where $i,j,k\in\{2, 3, 4\}$ or $i,j,k\in\{5,6,7\}$ and $i\neq j\neq k$. Let
$\mathbf{q}$ be any point in $S_{s,t>0}$ such that $q_i = p_i$, $q_j = p_j$, and $q_k = p_k$.  Then
$\mathbf{q}\in S_{s,t>0}-T$ as well. Since $S^2$ is path-connected, there is a path in $S_{s,t>0}$ from
$\mathbf{p}$ to $\mathbf{q}$ where every point on the path has $p_i$ as the $i$th coordinate. Hence the entire path is in
$S_{s,t>0}-T$.   

Now let $\mathbf{p}$ and $\mathbf{q}$ be any two points in
$S_{s,t>0}-T$.  As before, for some $2\leq i\leq 7$, $p_i \neq \pm
cb_{2i}$. Without loss of generallity suppose 
$p_2\neq \pm cb_{22}$.  We split the argument into two cases, one in
which $q_5 = \pm cb_{25}$, $q_6 = \pm cb_{26}$ and $q_7=\pm cb_{27}$ and the second with
$q_5 \neq \pm cb_{25}$, $q_6 \neq \pm cb_{26}$ or $q_7\neq \pm cb_{27}$.  

If $q_5 = \pm cb_{25}$, $q_6 = \pm cb_{26}$ and $q_7=\pm cb_{27}$ , then $q_2\neq \pm cb_{22}$, $q_3\neq \pm cb_{23}$ or $q_4\neq \pm cb_{24}$. Without loss of generality, suppose $q_2\neq \pm cb_{22}$.
Since $S^2$ is continuous and $S_{s,t>0}\cap T$ is discrete, there
exist $s_5, s_6,s_7$ such that $\lambda_5s_5^2+ \lambda_6s_6^2+ \lambda_7s_7^2 = t$ and  $s_5 \neq \pm cb_{25}$, $s_6 \neq \pm cb_{26}$ or $s_7\neq \pm cb_{27}$.  Then, by the argument
above, there is a path in $S_{s,t>0}-T$ from 
$\mathbf{p}$ to $(p_1, p_2, p_3, p_4, s_5, s_6, s_7)$ and  from 
$(p_1, p_2, p_3, p_4, s_5, s_6, s_7)$ to $(q_1, q_2, q_3, q_4, s_5, s_6, s_7)$.  
Since $q_2\neq \pm cb_{22}$, there exists a path
from $(q_1, q_2, q_3, q_4, s_5, s_6, s_7)$ to $\mathbf{q}$.  These paths combine
to give a path from $\mathbf{p}$ to $\mathbf{q}$ in $S_{s,t>0}-T$.

In the second case, $q_5 \neq \pm cb_{25}$, $q_6 \neq \pm cb_{26}$ or $q_7\neq \pm cb_{27}$. Using the argument above, there is a path from 
$\mathbf{p}$ to $(p_1, p_2, p_3, p_4, q_5, q_6, q_7)$ and a path from 
$(p_1, p_2, p_3, p_4, q_5, q_6, q_7)$ to $\mathbf{q}$ both of which are in
$S_{s,t>0}-T$.

We have now shown that $S_{s,t}-T$ is path connected since there exists a path between any $\mathbf{p}$ and $\mathbf{q}$ in $S_{s,t}-T$. Then $\GST_7$ is connected by Lemma 2.2.

\section{Applications of this Model}
In \cite{taylor}, Steel and Taylor present this topic in the context of a biological model, in addition to the context of a drinking game. In the biological example, the probability of the success of pollination for a single flower is dependent upon the time at which it blooms as compared to the time at which neighboring flowers bloom. In this section we offer several additional scenarios for which this model might be useful.

The first scenario we present addresses cheater organisms within populations. Plants and animals put energy toward building defenses against predators and diseases. Genetic mutations that lead to a lack of or diminished defenses can be advantageous to individuals, called ``cheaters", as long as enough of the surrounding organisms display defenses in order to prevent an epidemic or deter predators. In this scenario, a cause $C_i$ for an individual can be thought of as either building defenses toward a given threat, or having the ``cheater" genotype. The corresponding effect $E_i$ is 0 if individual $i$ has a less advantageous genotype and 1 if individual $i$ exhibits a more advantageous genotype. The probability of the effect $E_i$ then depends on the number of individuals with the same genotype as individual $i$.

A second scenario relates employee productivity to work location. Suppose that an $n$-person team works in a division of Company X. Team members can choose between working collaboratively in a large room, or individually in their offices. Depending on the number of people who choose to work in the large room, it can be a more or less productive environment than working in the offices. With an ideal number of people in the large room, team members work collaboratively and help each other stay focused. Too many people in the large room, however, can lead to distraction and poor time management. The cause, $C_i$, in this scenario is 0 if person $i$ chooses to work in his office and 1 if person $i$ chooses to work in the large room on a given day. The effect $E_j$ is 0 or 1 based on whether person $j$ completes her work for the day.

A third example models the commute to work for a group of people. Suppose $n$ people have similar daily commutes. Each day, each of the $n$ people choose whether to drive or take public transportation to work. If a large portion of the $n$ people choose to drive, then the traffic is much slower, and the bus, which travels in its own lane, is a faster option. If a large number of people choose to take the bus, however, the bus must make more stops and takes more time at each stop, making driving a faster choice.
In this scenario, person $i$'s choice whether to drive or take public transportation represents the cause $C_i$. 
An effect, $E_i$, is 0 if person $i$ does not arrive to work on time and 1 if person $i$ does arrive on time. Thus, the probability of the effect is determined by the number of people that choose the same method of transportation as person $i$.

\end{document}